\documentclass[11pt]{article}

\usepackage[a4paper,margin=1in]{geometry}
\usepackage{amsmath,amssymb,amsthm,mathtools}
\usepackage{array,booktabs,enumitem}
\usepackage[T1]{fontenc}
\usepackage{lmodern}
\usepackage{microtype}

\usepackage{hyperref}
\usepackage{tikz}
\usepackage{subcaption}
\hypersetup{
	colorlinks=true,
	linkcolor=blue,
	filecolor=magenta,
	urlcolor=cyan,
	citecolor=red
}
\newtheorem{definition}{Definition}[section]

\newtheorem{remark}[definition]{Remark}
\newtheorem{lemma}[definition]{Lemma}
\newtheorem{proposition}{Proposition}
\newtheorem{theorem}[proposition]{Theorem}

\newcommand{\Hcal}{\mathcal H}
\newcommand{\Spec}{\operatorname{Spec}}

\newcommand{\Z}{\mathbb Z}
\newcommand{\N}{\mathbb N}
\newcommand{\R}{\mathbb R}
\newcommand{\I}{\mathbf I}
\newcommand{\J}{\mathbf J}
\newcommand{\Cnk}{C_n^{[k]}}
\newcommand{\ZZ}{\mathbb Z}
\newcommand{\Q}{\mathbb Q}
\DeclareMathOperator{\Gal}{Gal}

\title{Solution to a conjecture on integral uniform hypercycles}

\author{Joyentanuj Das\thanks{Department of Mathematics, SRM Institute of Science and Technology Chennai, Tamil Nadu, India. Email: joyentanuj@gmail.com, joyentad@srmist.edu.in.}\  \and Iswar Mahato \thanks{Department of Mathematics, SRM University-AP, Andhra Pradesh 522240, India. Email: iswarmahato02@gmail.com, iswar.m@srmap.edu.in.}}

\date{}
\begin{document}
\maketitle

\begin{abstract}
A hypergraph is said to be integral if all of its adjacency eigenvalues are integers. Recently, Portugal and Del-Vecchio in 
[\emph{Appl. Math. Comput.} 504: 129507 (2025)] studied the integral hypergraphs and gave a characterization of integral hypercycles in three particular cases: $3$-uniform, $4$-uniform and $5$-uniform hypercycles. In the same article, they conjectured that the $k$-uniform hypercycle on $n$ vertices $\Cnk$ is never integral for $k<n-1$, when $n>6$. In this article, we confirm this conjecture and prove that for $2\le k\le n-1$, $\Cnk$ is integral if and only if $k=n-1$ or $(n,k)\in\{(4,2),(6,2),(6,3),(6,4)\}$. The proof begins with computing the complete adjacency spectrum of $\Cnk$ and then uses Niven's theorem, a cyclotomic-unit lemma, an elementary property of Euler's totient function, and the Galois symmetry of cyclotomic fields to complete it. Our result gives a complete characterization of $k$-uniform integral hypercycles on $n$ vertices. 
\end{abstract}



\section{Introduction}
A hypergraph $\Hcal=(V,E)$ is given by a vertex set $V$ and an edge set $E=\{e:e\subseteq V\}$, whose elements are called (hyper)edges. A hypergraph is simple if it has no loops (edges with $|e|=1$) and no repeated edges. Throughout this article, we consider only simple hypergraphs. The hypergraph $\Hcal$ is \emph{$k$-uniform} if $|e|=k$ for every edge $e\in E(\Hcal)$. Therefore, a graph is a $2$-uniform hypergraph. The degree of a vertex $v\in V$, denoted by $d(v)$, is the number of edges that contain $v$. A hypergraph is $r$-regular if $d(v)=r$ for all $v\in V$.

For $n\in\N$, let $\Z_n=\{0,1,\ldots,n-1\}$ be the ring of integers modulo $n$. A $k$-uniform hypercycle on $n$ vertices, denoted by $C_n^{[k]}$, is the hypergraph with 
vertex set $V=\Z_n$ and edge set
\[
 E=\{e_j:j\in V\},
 \qquad \text{where} \qquad e_j=\{j,j+1,\ldots,j+k-1\},
\]
where all additions are taken modulo $n$. Note that $C_n^{[k]}$ is $k$-regular and $|V|=|E|=n$.

Recently, the spectral theory of hypergraphs has emerged as an active area of research in algebraic and spectral graph theory. There are two principal approaches to studying the spectra of hypergraphs: one based on tensors (or hypermatrices) and the other based on matrices. The tensor-based approach was developed systematically by Cooper and Dutle~\cite{CooperDutle}, who introduced a spectral theory for uniform hypergraphs through the adjacency hypermatrix. On the other hand, matrix-based approaches associate suitable matrices with a hypergraph and investigate their eigenvalues to reveal structural properties. Classical matrix-based treatments of hypergraph spectra can be found in Feng and Li~\cite{FengLi}, while more recent developments and systematic studies of connectivity matrices and their spectra are given in Banerjee~\cite{Banerjee}. Further developments involving adjacency matrices and related spectral parameters can be found in Cardoso et al.~\cite{CardosoDelVecchioPortugalTrevisan}. In this paper, we adopt the matrix-based spectral framework for hypergraphs.

In 1974, Harary and Schwenk proposed the following interesting question: ``Which graphs have integral spectra?'' The problem is elementary to state but difficult in general, and it
has led to a substantial literature devoted to particular graph classes and to structural and algebraic techniques for controlling the spectrum; see, for example, the survey of Balinska et al.~\cite{Balinska} and the standard texts on graph spectra~\cite{CvetkovicDoobSachs}.

Very recently, Portugal and Del-Vecchio \cite{PortugalDelVecchio} considered the corresponding question for hypergraphs and studied integral hypergraphs--hypergraphs whose all adjacency eigenvalues are integers. They constructed infinite families of integral uniform hypergraphs and obtained the following result about integrality of hypercycles:

\begin{theorem}\label{thm:Portugal_main}
Let $C_n^{[k]}$ be the $k$-uniform hypercycle on $n$ vertices. Then
\begin{enumerate}
    \item For $k=3$ and $n>4$, the hypercycle $C_n^{[3]}$ is integral if and only if $n=6$.
    \item For $k=4$ and $n>5$, the hypercycle $C_n^{[4]}$ is integral if and only if $n=6$.
    \item For $k=5$ and $n>6$, the $5$-uniform hypercycle $C_n^{[5]}$ is never integral.
\end{enumerate}
\end{theorem}

\noindent In the same article, the authors made the following conjecture: \emph{``The $k$-uniform hypercycle $\Cnk$ is never integral for $k<n-1$, when $n>6$''}.

In this article, we solve this conjecture affirmatively and prove the following result.

\begin{theorem}\label{thm:classification-new}
Let $n\ge3$ and $2\le k\le n-1$. Then $\Cnk$ is integral if and only if $k=n-1$ or $(n,k)\in\{(4,2),(6,2),(6,3),(6,4)\}.$
\end{theorem}

For the proof of Theorem \ref{thm:classification-new}, we first compute the explicit spectrum of the adjacency matrix of $\Cnk$ and then use Niven's theorem, a cyclotomic-unit lemma, an elementary property of Euler's totient function, and the Galois symmetry of cyclotomic fields to complete it. The proof is presented in section 3.

\section{Preliminaries}

In this section, we collect some basic definitions and discuss some preliminary results which will be used to prove our main results. First, let us give the definition of the adjacency matrix of a hypergraph. 

\begin{definition}
  Let $\Hcal$ be a hypergraph on $n$ vertices. Then the adjacency matrix of $\Hcal$, denoted by $A(\Hcal)$, is the $n\times n$ symmetric matrix with 
\[
 a_{ij}=\bigl|\{e\in E(\Hcal):v_i,v_j\in e\}\bigr|.
\]  
\end{definition}

If $\lambda_1,\lambda_2,\ldots,\lambda_t$ are all the distinct eigenvalues of $A(\Hcal)$, then the spectrum of $\Hcal$ is given by
\[\Spec(\Hcal)
 =\{(\lambda_1)^{m_1},(\lambda_2)^{m_2},\ldots,
          (\lambda_t)^{m_t}\},\]
where $m_1,\ldots,m_t$ are the corresponding multiplicities. Now, we recall the definition of a circulant matrix.

\begin{definition}
A circulant matrix is a square matrix in which all row vectors are composed of the same elements, and each row vector is rotated one element to the right relative to the previous row vector.
\end{definition}

A circulant matrix is specified by its first row. If
$[0,a_1,a_2,\ldots,a_{n-1}]$ is the first row, the its eigenvalues are
\[\lambda_j=\sum_{k=1}^{n-1}a_k\omega^{kj}
 \quad \text{for} \quad j=0,\ldots,n-1,
~~ \text{where} ~~
 \omega=\exp\left(\frac{2\pi i}{n}\right).\]

\noindent Note that the adjacency matrix of $C_n^{[k]}$ is circulant.

We will use a few standard facts from elementary number theory and cyclotomic fields. We record them here in the precise forms needed below. Let $\varphi$ denote Euler\'s totient function. If $\eta$ is a primitive $m$-th root of unity, then $\Q(\eta)$ denotes the $m$-th cyclotomic field and $\ZZ[\eta]$ be the ring generated over $\ZZ$ by $\eta$. We shall use the standard description
$$
\Gal(\Q(\eta)/\Q)\cong(\Z/m\Z)^\times,
$$
with $t\in(\Z/m\Z)^\times$ acting by $\eta\mapsto\eta^t$; see Washington~\cite{Washington} for the notations and terminologies.

\begin{lemma}[Niven's theorem]\label{lem:niven-new}
Let $\theta\in\R$ be a rational multiple of $\pi$. If $\cos\theta\in\Q$, then
$$ \cos\theta\in \left\{ -1,-\frac12,0,\frac12,1\right\}.$$
\end{lemma}

\begin{remark}
Lemma~\ref{lem:niven-new} is the cosine form of Niven's theorem; see Niven~\cite{Niven}. Equivalently, if $\frac{\theta}{\pi}\in\Q$ and $\cos\theta$ is rational, then the only possible rational values are $-1,-\frac12,0,\frac12,1.$
\end{remark}

\begin{lemma}\label{lem:unit-new}
For $m\ge2$, let $\eta$ be a primitive $m$-th root of unity, and let $r\in\Z$ satisfy $\gcd(r,m)=1.$ If
$$\left|\frac{1-\eta^r}{1-\eta}\right|^2\in\Q,$$ then $$ r\equiv\pm1\pmod m.$$
\end{lemma}

\begin{proof}
Set $$\beta=\frac{1-\eta^r}{1-\eta}=1+\eta+\eta^2+\cdots+\eta^{r-1}.$$ Therefore, $\beta\in\Z[\eta].$ Since $\gcd(r,m)=1$, there exists an integer $s$ such that $rs\equiv1\pmod m.$ Hence, $\eta^{rs}=\eta$ and 
$$ \beta^{-1}=\frac{1-\eta}{1-\eta^r}=\frac{1-(\eta^r)^s}{1-\eta^r} \in\Z[\eta].$$ 
Therefore, $\beta$ is a unit in $\Z[\eta]$. Since the  complex conjugate $\overline{\beta}$ of $\beta$ is also a unit in $\Z[\eta]$, 
$$\gamma:=\beta\overline{\beta} =\left|\frac{1-\eta^r}{1-\eta}\right|^2 $$
is also a unit of $\Z[\eta]$. By hypothesis, $\gamma\in\Q$. Since $\gamma$ is also an algebraic integer, we have $$\gamma\in\Z[\eta]\cap\Q=\Z.$$ Moreover, $\gamma$ is a unit, so its inverse is also an algebraic integer. Since $\gamma^{-1}\in\Q$, the same argument gives $\gamma^{-1}\in\Z.$ Thus, $\gamma$ is a positive rational integer whose inverse is also an integer. Consequently, $\gamma=1.$ Therefore, $|1-\eta^r|=|1-\eta|.$

Let $\eta=e^{2\pi i s/m},$ where $\gcd(s,m)=1.$ Then $|1-\eta|=2\left|\sin\frac{\pi s}{m}\right|$ and $|1-\eta^r|=2\left|\sin\frac{\pi rs}{m}\right|.$ Hence, $$\left|\sin\frac{\pi rs}{m}\right| =\left|\sin\frac{\pi s}{m}\right|.$$ 
Using the relation
$$\sin^2 x=\sin^2 y\quad\Longleftrightarrow\quad x\equiv\pm y\pmod{\pi}, $$ we obtain
$$
\frac{\pi rs}{m}
\equiv
\pm\frac{\pi s}{m}
\pmod{\pi}.
$$
Therefore, $rs\equiv\pm s\pmod m.$ This implies that $r\equiv\pm1\pmod m$  as $\gcd(s,m)=1$. This completes the proof.
\end{proof}

\begin{remark}
The fact that $\frac{1-\eta^r}{1-\eta}$ is a cyclotomic unit when $\gcd(r,m)=1$ is standard in the theory of cyclotomic fields; see, for example, Washington~\cite{Washington}. The lemma above is the particular consequence of this standard fact that will be needed in the proof of the main classification theorem.
\end{remark}

\begin{lemma}\label{lem:phi-new}
Let $\nu,n\in\N$ satisfy $\nu\mid n$ and $\nu<n.$ If $\varphi(\nu)=\varphi(n),$ then $n=2\nu$ and $\nu$ is odd.
\end{lemma}

\begin{proof}
Let the prime factorizations of $n$ and $\nu$ be
$$
n=\prod_p p^{a_p},
\qquad
\nu=\prod_p p^{b_p},
$$
where $0\le b_p\le a_p$ for every prime $p$, because $\nu\mid n$. Recall that
$$\varphi(n)=\prod_{p:\,a_p>0}p^{a_p-1}(p-1),$$ and similarly for $\varphi(\nu)$. Hence,
$$
\frac{\varphi(n)}{\varphi(\nu)}
=
\prod_{p:\,b_p>0}p^{a_p-b_p}
\prod_{p:\,b_p=0<a_p}p^{a_p-1}(p-1).
$$
Observe that every factor in the product is a positive integer. Since $\frac{\varphi(n)}{\varphi(\nu)}=1$, every factor
must be equal to $1$. If $b_p>0$, then $p^{a_p-b_p}=1,$ and hence $a_p=b_p.$ Thus, no prime dividing $\nu$ can occur in $n$ with a larger exponent. Now suppose that $b_p=0<a_p.$ Then $p^{a_p-1}(p-1)=1.$ Since both factors are positive integers, we must have $p-1=1$ and $p^{a_p-1}=1.$ Thus, $p=2$ and $a_p=1.$ Since $\nu<n$, at least one prime occurs in $n$ but not in $\nu$. The preceding argument shows that this prime must be $2$, and it occurs only to the first power. Hence, $n=2\nu.$ Since the prime $2$ does not divide $\nu$, we conclude that $\nu$ is odd.
\end{proof}

\begin{lemma}\label{lem:galois-new}
Let $\zeta=e^{2\pi i/n}, \alpha_j =\sum_{m=0}^{k-1}\zeta^{mj},$ and $\mu_j=\alpha_j\overline{\alpha_j}$ for $j=1,\ldots,n-1$.
If $\gcd(j,n)=\gcd(j',n)$ and $\mu_j\in\Q,$ then $\mu_{j'}=\mu_j.$
\end{lemma}

\begin{proof}
Let $d=\gcd(j,n)=\gcd(j',n).$ Then the residues $\dfrac{j}{d}$ and $\dfrac{j'}{d}$ are both relatively prime to $\dfrac{n}{d}.$ Equivalently, there exists an integer $t$ satisfying $\gcd(t,n)=1$ and $j'\equiv tj\pmod n.$ We now use the standard description of the Galois group of the cyclotomic field: $\operatorname{Gal}(\Q(\zeta)/\Q) \cong (\Z/n\Z)^\times,$ where the automorphism corresponding to $t$ is determined by $\sigma_t(\zeta)=\zeta^t$ (for reference see Washington~\cite{Washington}). Applying $\sigma_t$ to $\alpha_j$, we obtain $$\sigma_t(\alpha_j) = \sigma_t\left(\sum_{m=0}^{k-1}\zeta^{mj}\right) =\sum_{m=0}^{k-1}\sigma_t(\zeta)^{mj} =\sum_{m=0}^{k-1}\zeta^{tmj} =\sum_{m=0}^{k-1}\zeta^{mj'} =\alpha_{j'}. $$ The last equality follows from $j'\equiv tj\pmod n.$

Since $\sigma_t$ fixes the rational numbers and commutes with complex conjugation, we also have $\sigma_t(\overline{\alpha_j}) =\overline{\sigma_t(\alpha_j)}=\overline{\alpha_{j'}}.$ Consequently, $\sigma_t(\mu_j)=\sigma_t(\alpha_j\overline{\alpha_j})=\alpha_{j'}\overline{\alpha_{j'}}=\mu_{j'}.$ Since $\mu_j\in\Q$, every automorphism of $\Q(\zeta)/\Q$ fixes every rational number and hence $\sigma_t(\mu_j)=\mu_j.$ Thus, $\mu_{j'}=\mu_j.$
\end{proof}

\begin{remark}
Lemma~\ref{lem:galois-new} is a direct application of the standard Galois action on a cyclotomic field. Its role is to show that, whenever one of the quantities $\mu_j$ is rational, the values $\mu_j$ depend only on $\gcd(j,n)$.
\end{remark}

\section{Main results}
In this section, we give the proof of Theorem \ref{thm:classification-new}. For this, first we compute the eigenvalues of $\Cnk$ for $2\le k\le n-1$. Throughout this section, we define $\zeta=e^{2\pi i/n}.$

\begin{theorem}\label{thm:spectrum-new}
Let $\Cnk$ be the $k$-uniform hypercycle on $n$ vertices with $2\le k\le n-1$. If $\lambda_0,\lambda_1,\hdots, \lambda_{n-1}$ are the  eigenvalues of $A(\Cnk)$, then 
$$
\lambda_j
=
\left|\sum_{m=0}^{k-1}\zeta^{mj}\right|^2-k
=
\frac{\sin^2(\pi kj/n)}{\sin^2(\pi j/n)}-k,
\qquad j=1,\ldots,n-1,
$$
while $\lambda_0=k(k-1).$
\end{theorem}

\begin{proof}
Let $V=\Z_n=\{0,1,\ldots,n-1\},$ and recall that
$$
E(\Cnk)=\{e_t:t\in\Z_n\},
\qquad
e_t=\{t,t+1,\ldots,t+k-1\},
$$
here all additions are taken modulo $n$. We shall obtain the spectrum of the adjacency matrix $A(\Cnk)$ by passing first to the vertex--edge incidence matrix of $\Cnk$.

Let $M=(m_{v,t})_{v,t\in\Z_n}$ be the vertex--edge incidence matrix of $\Cnk$, defined by
$$
m_{v,t}
=
\begin{cases}
1,&v\in e_t,\\
0,&v\notin e_t.
\end{cases}
$$
Since $2\le k\le n-1$, the edges $e_0,e_1,\ldots,e_{n-1}$ are pairwise distinct. 

We first determine the matrix $MM^{\mathsf T}$. For distinct vertices $u,v\in\Z_n$, the $(u,v)$-entry of $MM^{\mathsf T}$ is
$$
(MM^{\mathsf T})_{uv}
=
\sum_{t\in\Z_n}m_{u,t}m_{v,t}.
$$
The summand $m_{u,t}m_{v,t}$ is equal to $1$ exactly when both $u$ and $v$ belong to the edge $e_t$. Therefore, by the definition of the adjacency matrix,
$$
(MM^{\mathsf T})_{uv}
=
\bigl|\{e\in E(\Cnk):u,v\in e\}\bigr|
=
a_{uv},
\qquad u\ne v.
$$
For $u=v$, we have
$$
(MM^{\mathsf T})_{uu}
=
\sum_{t\in\Z_n}m_{u,t}^2
=
\sum_{t\in\Z_n}m_{u,t}
=
d(u)
=
k.
$$
Since the diagonal entries of $A(\Cnk)$ are zero, we have
$$
MM^{\mathsf T}=A(\Cnk)+kI,
$$
and hence
\begin{equation}\label{eq:MMt-new}
A(\Cnk)=MM^{\mathsf T}-kI.
\end{equation}

Thus, if $\mu$ is an eigenvalue of $MM^{\mathsf T}$, then $\mu-k$ is an eigenvalue of $A(\Cnk)$. It remains to determine the eigenvalues of $MM^{\mathsf T}$.

Let $\zeta=e^{2\pi i/n}.$ We claim that $M$ is a circulant matrix. Indeed, for $v,t\in\Z_n$, $v\in e_t$ if and only if $v\in\{t,t+1,\ldots,t+k-1\},$ which, modulo $n$, is equivalent to $t-v\in\{0,-1,\ldots,-(k-1)\}.$ Consequently, if we define
$$
c_d=
\begin{cases}
1,&d\in\{0,-1,\ldots,-(k-1)\}\pmod n,\\
0,&\text{otherwise},
\end{cases}
$$
then $m_{v,t}=c_{t-v}.$ Thus the entries of $M$ depend only on $t-v$ modulo $n$, which is precisely the defining property of a circulant matrix.

For $j=0,1,\ldots,n-1$, consider the vector
$$
u_j
=
\left(
1,\zeta^j,\zeta^{2j},\ldots,\zeta^{(n-1)j}
\right)^{\mathsf T}.
$$
The vectors $u_0,u_1,\ldots,u_{n-1}$ form a basis of $\mathbb C^n$ and, since $M$ is circulant, they form an eigenbasis for $M$. We now compute the eigenvalue corresponding to $u_j$. For $v\in\Z_n$, $$ (Mu_j)_v = \sum_{t\in\Z_n}m_{v,t}\zeta^{tj}.$$ Using $m_{v,t}=c_{t-v}$ and putting $d=t-v$, we obtain $$ (Mu_j)_v = \sum_{d\in\Z_n}c_d\zeta^{(v+d)j} = \zeta^{vj}\sum_{d\in\Z_n}c_d\zeta^{dj}.$$
By the definition of $c_d$, the only nonzero terms occur for $d=0,-1,\ldots,-(k-1).$ Hence
$$
(Mu_j)_v
=
\zeta^{vj}
\sum_{m=0}^{k-1}\zeta^{-mj}.
$$
Therefore, $Mu_j = \widehat c(j)\,u_j,$ where $ \widehat c(j)=\sum_{m=0}^{k-1}\zeta^{-mj}.$ Since $M$ is a real matrix, $M^{\mathsf T}=\overline{M}^{\,*},$ and the vector $u_j$ is also an eigenvector of $M^{\mathsf T}$, with eigenvalue $\overline{\widehat c(j)}$. Thus,
$$ MM^{\mathsf T}u_j = |\widehat c(j)|^2u_j. $$ Hence, the eigenvalues of $MM^{\mathsf T}$ are
$$ |\widehat c(j)|^2 = \left|\sum_{m=0}^{k-1}\zeta^{-mj}\right|^2=\left|\sum_{m=0}^{k-1}\zeta^{mj}\right|^2. $$
 Therefore, by \eqref{eq:MMt-new}, it follows that the eigenvalues of $A(\Cnk)$ are $$ \lambda_j=\left|\sum_{m=0}^{k-1}\zeta^{mj}\right|^2-k, \qquad j=0,1,\ldots,n-1.$$

We now simplify this expression. First, consider $j=1,\ldots,n-1$. Since $\zeta^j\ne1$, the finite geometric-series identity gives
$$
\sum_{m=0}^{k-1}\zeta^{mj}
=
\frac{1-\zeta^{kj}}{1-\zeta^j}.
$$
Therefore,
$$
\left|\sum_{m=0}^{k-1}\zeta^{mj}\right|^2
=
\frac{|1-\zeta^{kj}|^2}{|1-\zeta^j|^2}.
$$
Now, recall the following elementary identity, which will be used in the subsequent steps.
\begin{equation}\label{eq:e^i_identity}
|1-e^{i\theta}|^2
=
(1-e^{i\theta})(1-e^{-i\theta})
=
2-2\cos\theta
=
4\sin^2\left(\frac{\theta}{2}\right).
\end{equation}
Since $\zeta^{kj}=e^{2\pi i kj/n}$ and $\zeta^j=e^{2\pi i j/n},$ by (\ref{eq:e^i_identity}) we obtain
\begin{align*}
|1-\zeta^{kj}|^2
=
4\sin^2\left(\frac{\pi kj}{n}\right)\qquad \text{and} \qquad
|1-\zeta^j|^2
=
4\sin^2\left(\frac{\pi j}{n}\right).
\end{align*}
Thus,
$$
\left|\sum_{m=0}^{k-1}\zeta^{mj}\right|^2
=
\frac{\sin^2(\pi kj/n)}
{\sin^2(\pi j/n)}.
$$
Since $1\le j\le n-1$, we have
$$
\sin\left(\frac{\pi j}{n}\right)\ne0,
$$
so the denominator is nonzero. Consequently,
$$
\lambda_j
=
\frac{\sin^2(\pi kj/n)}
{\sin^2(\pi j/n)}
-k
\qquad \text{for} \qquad j=1,\ldots,n-1.
$$
Finally, consider $j=0$. In this case $\zeta^0=1$, and hence $$ \sum_{m=0}^{k-1}\zeta^{m\cdot0} =\sum_{m=0}^{k-1}1=k.$$ Therefore, $\lambda_0 =|k|^2-k=k^2-k = k(k-1)$. This proves the stated formula for all eigenvalues of $A(\Cnk)$.
\end{proof}

\begin{remark}\label{rem:first-row-new}
The adjacency spectrum of $\Cnk$ given in Theorem~\ref{thm:spectrum-new} is valid for every $2\le k\le n-1$ and does not require the auxiliary condition $n\ge2k-1$. For $d\in\Z_n$, the first row of $A(\Cnk)$ satisfies
$$
a_{0,d}=\max(0,k-d)+\max(0,k-n+d).
$$
When $n\ge2k-1$, this reduces to
$[0,k-1,k-2,\ldots,2,1,0,\ldots,0,1,2,\ldots,k-2,k-1]$ given in Remark 3.6 of \cite{PortugalDelVecchio}.
Moreover,
$$
\mu_j
=\left|\sum_{m=0}^{k-1}\zeta^{mj}\right|^2
=k+2\sum_{m=1}^{k-1}(k-m)
\cos\left(\frac{2\pi mj}{n}\right),
$$
which implies that
$$
\lambda_j
=2\sum_{m=1}^{k-1}(k-m)
\cos\left(\frac{2\pi mj}{n}\right).
$$
From this, we can recover the explicit spectra of $C_n^{[3]},C_n^{[4]}$ and $C_n^{[5]}$ recorded in Proposition 4, 6, and 8 of \cite{PortugalDelVecchio}.
\end{remark}

Now, we are ready to give the proof of Theorem \ref{thm:classification-new}.\\

\noindent \textbf{Proof of Theorem \ref{thm:classification-new}}: First, let us assume that $k=n-1$. Then every pair of distinct vertices in $\Cnk$ belong to exactly $n-2$ hyperedges, and hence
$$A(\Cnk)=(n-2)(\J-\I).
$$
Consequently,
$$
\Spec(\Cnk)
=
\left\{((n-1)(n-2))^1, (-(n-2))^{n-1}\right\}.
$$
Thus, $\Cnk$ is integral for $k=n-1$. For the four remaining pairs, by Theorem~\ref{thm:spectrum-new}, it follows that
$$\Spec(C_4^{[2]})=\{2^1,0^2,(-2)^1\},\qquad \Spec(C_6^{[2]})=\{2^1,1^2,(-1)^2,(-2)^1\},$$
$$\Spec(C_6^{[3]})=\{6^1,1^2,(-2)^1,(-3)^2\}\quad \text{ and } \quad \Spec(C_6^{[4]})=\{12^1,(-1)^2,(-3)^2,(-4)^1\}.
$$
Thus, all four of these hypercycles are integral. It remains to prove that there are no other integral uniform hypercycles for $2\le k\le n-1$. Assume that $\Cnk$ is integral for $2\le k\le n-2.$ Therefore, by Theorem \ref{thm:spectrum-new}, it follows that
$$\mu_j:=\lambda_j+k=\left|\sum_{m=0}^{k-1}\zeta^{mj}\right|^2=\left|\frac{1-\zeta^{kj}}{1-\zeta^j}\right|^2,\qquad j=1,\ldots,n-1,$$ are all integers, and in particular they are rational.

Let $g=\gcd(n,k), \nu=\dfrac{n}{g},$ and $\kappa=\dfrac{k}{g}.$ Then $n=g\nu, k=g\kappa,$ and $\gcd(\kappa,\nu)=1.$ We first show that $g>1$. Suppose, to the contrary, that $\gcd(n,k)=1.$ Then $\mu_1 = \left|\frac{1-\zeta^k}{1-\zeta}\right|^2$ is an integer. Since $\gcd(k,n)=1$, applying Lemma~\ref{lem:unit-new} with $m=n, \eta=\zeta, r=k$, we have  $k\equiv\pm1\pmod n$, which is not possible because $2\le k\le n-2$. Therefore, $g>1.$

Clearly $\nu\ne1$, since $\nu=1$ would imply $g=n$ and hence $n\mid k$, contrary to $k<n$. Thus, $\nu\ge2.$ First, consider the case $\nu=2.$ Then $n=2g.$ Since $g\mid k$, we may write $k=gq$ for some positive integer $q$. Since $k<n=2g$, we must have $q=1$. Hence, $k=g=\dfrac n2.$ Therefore, $\zeta^k=\zeta^{n/2}=-1,$ and using Theorem~\ref{thm:spectrum-new} we have $$
\mu_1=\left|\frac{1-\zeta^k}{1-\zeta}\right|^2=\left|\frac{2}{1-\zeta}\right|^2.$$ Further using the identity
$$|1-e^{i\theta}|^2=4\sin^2\left(\frac{\theta}{2}\right),$$ we have
$$\mu_1=\frac{1}{\sin^2(\pi/n)}=\frac{2}{1-\cos(2\pi/n)}.$$ Since $\mu_1\in\mathbb Z$, it follows that
$$
\cos\left(\frac{2\pi}{n}\right)
=
1-\frac{2}{\mu_1}
\in\Q.
$$
Since $2\pi/n$ is a rational multiple of $\pi$, by Lemma~\ref{lem:niven-new} it follows that $\cos\left(\frac{2\pi}{n}\right) \in \left\{
-1,-\frac12,0,\frac12,1\right\}.$ Since $n\ge3$ is finite and even, the corresponding possibilities give $n\in\{4,6\}.$ As $k=n/2$, we obtain $(n,k)\in\{(4,2),(6,3)\}.$

Now, let us assume that $\nu\ge3.$ We will show that $g=2$ and $\nu$ is odd. Consider the reduction map
$$\rho:(\Z/n\Z)^\times\longrightarrow(\Z/\nu\Z)^\times, \qquad [j]_n\longmapsto[j]_\nu. $$
We claim that $\rho$ is a bijection. First, let us show that $\rho$ is injective.  Assume that $\rho([j_1]_n)=\rho([j_2]_n)$ for two distinct elements $[j_1]_n,[j_2]_n\in(\mathbb Z/n\mathbb Z)^\times$. Thus $j_1\equiv j_2\pmod{\nu},$ and $\gcd(j_1,n)=\gcd(j_2,n)=1.$ We will show that $j_1\equiv j_2\pmod n$. Recall that
$$
\mu_j = \left|\sum_{m=0}^{k-1}\zeta^{mj}\right|^2 =\frac{\sin^2(\pi kj/n)}{\sin^2(\pi j/n)}
\qquad \text{for} \qquad
j=1,\ldots,n-1.
$$
Since $\frac{k}{n}=\frac{g\kappa}{g\nu}=\frac{\kappa}{\nu},$ the congruence $j_1\equiv j_2\pmod{\nu}$ implies $\frac{kj_1}{n}=\frac{\kappa j_1}{\nu}\equiv\frac{\kappa j_2}{\nu}=\frac{kj_2}{n}\pmod 1.$ Hence
$$\sin^2\left(\frac{\pi kj_1}{n}\right)=\sin^2\left(\frac{\pi kj_2}{n}\right).$$
Moreover, these two quantities are nonzero. Indeed, $\gcd(\kappa,\nu)=1$ and $\gcd(j_i,\nu)=1$, so $\gcd(\kappa j_i,\nu)=1$, and therefore $\nu\nmid\kappa j_i$. We now claim that $\mu_{j_1}\ne\mu_{j_2}.$ Suppose, to the contrary, that $\mu_{j_1}=\mu_{j_2}.$ Therefore,
$$\sin^2\left(\frac{\pi j_1}{n}\right)=\sin^2\left(\frac{\pi j_2}{n}\right),$$ which implies that $j_1\equiv\pm j_2\pmod n.$ The congruence $j_1\equiv j_2\pmod n$ is precisely the desired conclusion. Suppose instead that $j_1\equiv-j_2\pmod n.$ Since $\nu\mid n$, reducing this congruence modulo $\nu$ gives $j_1\equiv-j_2\pmod\nu.$ On the other hand, from $\rho([j_1]_n)=\rho([j_2]_n)$ we have $j_1\equiv j_2\pmod\nu.$ Therefore $j_1\equiv-j_1\pmod\nu, $ and hence $2j_1\equiv0\pmod\nu.$ Since $\gcd(j_1,\nu)=1$, it follows that $\nu\mid2$, contradicting $\nu\ge3$. Consequently, $\mu_{j_1}\ne\mu_{j_2}.$ On the other hand, $\Cnk$ is integral, so every $\mu_j=\lambda_j+k$ is an integer and hence rational. Since $\gcd(j_1,n)=\gcd(j_2,n)=1,$ from Lemma~\ref{lem:galois-new} we have $\mu_{j_1}=\mu_{j_2},$ a contradiction. Therefore our assumption that $[j_1]_n$ and $[j_2]_n$ are distinct is impossible. Hence $[j_1]_n=[j_2]_n,$ and thus $\rho$ is injective.

Now, we will show that the map $\rho$ is surjective. To see this let $[a]_\nu\in(\Z/\nu\Z)^\times.$ Choose an integer representative $a$ with $\gcd(a,\nu)=1.$ Let $p$ be a prime divisor of $n$ which does not divide $\nu$. Since $\gcd(\nu,p)=1$, the Chinese remainder theorem allows us to impose simultaneously the congruences $$x\equiv a\pmod\nu \text{ and } x\equiv1\pmod p $$ for all such primes $p$. The resulting residue class $[x]_n$ is coprime to every prime divisor of $n$, hence $[x]_n\in(\Z/n\Z)^\times,$ and it reduces to $[a]_\nu$. Thus, $\rho$ is surjective.


Since \(\rho\) is a bijection between \((\mathbb Z/n\mathbb Z)^\times\) and \((\mathbb Z/\nu\mathbb Z)^\times\), their cardinalities are same, that is, \(\varphi(n)=\varphi(\nu)\).  Since $\nu\mid n$ and $\nu<n,$ Lemma~\ref{lem:phi-new} yields $n=2\nu$ and $\nu$ is odd. Since $n=g\nu,$ we conclude that $g=2.$ Thus $n=2\nu, k=2\kappa, \gcd(\kappa,\nu)=1,$ and $\nu\ge3$ is odd. In particular, $\gcd(k,\nu)=1.$

Let $\eta=\zeta^2.$ Since $\zeta$ is a primitive $n$-th root of unity and $n=2\nu$ with $\nu$ odd, $\eta$ is a primitive $\nu$-th root of unity. Therefore,
$$
\mu_2
=
\left|
\frac{1-\zeta^{2k}}{1-\zeta^2}
\right|^2
=
\left|
\frac{1-\eta^k}{1-\eta}
\right|^2.
$$
Since $\Cnk$ is integral, $\mu_2\in\mathbb Z.$ Again, since $\gcd(k,\nu)=1$, Lemma~\ref{lem:unit-new} applied with $m=\nu, \eta=\zeta^2$ and $r=k$ gives $k\equiv\pm1\pmod\nu.$ Since $0<k<2\nu$ and $k$ is even while $\nu$ is odd, the congruence $k\equiv1\pmod\nu$ can only give $k=\nu+1,$ while $k\equiv-1\pmod\nu$ can only give $k=\nu-1.$ Therefore, $k=\nu\pm1.$

It remains to determine $\nu$. Since $n=2\nu$ and $k=\nu\pm1,$ we have $\frac{\pi k}{n} =\frac{\pi}{2}\pm\frac{\pi}{n}.$ Therefore, $\sin^2\left(\frac{\pi k}{n}\right) =\cos^2\left(\frac{\pi}{n}\right).$ Using the formula for $\mu_1$, we obtain
$$
\mu_1
=
\frac{\sin^2(\pi k/n)}
{\sin^2(\pi/n)}
=
\frac{\cos^2(\pi/n)}
{\sin^2(\pi/n)}
=
\cot^2\left(\frac{\pi}{n}\right).
$$
Using the identity
$\cot^2 x=\frac{1+\cos(2x)}{1-\cos(2x)}$, we have
$\mu_1=\frac{1+\cos(2\pi/n)}{1-\cos(2\pi/n)}.$
Since $\mu_1$ is an integer, we have $\cos\left(\frac{2\pi}{n}\right) =\frac{\mu_1-1}{\mu_1+1}
\in\Q.$ Again, $2\pi/n$ is a rational multiple of $\pi$, so by Lemma~\ref{lem:niven-new} it follows that $\cos\left(\frac{2\pi}{n}\right)
\in\left\{-1,-\frac12,0,\frac12,1\right\}.$ Since $n=2\nu\ge6,$ the only possible value of $n$ is $n=6.$ Thus, $\nu=3,$ and hence
$k=\nu\pm1\in\{2,4\}.$ This gives $(n,k)=(6,2)$ and $(n,k)=(6,4).$ Thus, by combining these with the cases obtained when $\nu=2$, namely $(4,2)$ and $(6,3),$ we conclude that, under the assumption $2\le k\le n-2$, the only integral hypercycles are
$$
C_4^{[2]},\quad
C_6^{[2]},\quad
C_6^{[3]}\quad \text{and} \quad
C_6^{[4]}.
$$
Together with the already established case $k=n-1$, this completes the proof.


\end{document}